\documentclass[a4paper]{amsart}
\usepackage{amssymb, hyperref, aliascnt}
\usepackage{tikz-cd}
\usepackage[capitalise, nameinlink, noabbrev, nosort]{cleveref} 

\theoremstyle{plain}
\newtheorem{lma}{Lemma}[section]
\crefname{lma}{Lemma}{Lemmata}

\crefname{thm}{Theorem}{Theorems}
\newtheorem{cor}[lma]{Corollary}
\crefname{cor}{Corollary}{Corollaries}

\crefname{prp}{Proposition}{Propositions}

\theoremstyle{definition}

\crefname{pgr}{Paragraph}{Paragraphs}

\crefname{dfn}{Definition}{Definitions}

\theoremstyle{remark}

\crefname{rmk}{Remark}{Remarks}
\newtheorem{exa}[lma]{Example}
\crefname{exa}{Example}{Examples}

\theoremstyle{plain}

\newcounter{IntroCount}

\newtheorem{thmIntro}[IntroCount]{Theorem}
\crefname{thmIntro}{Theorem}{Theorems}

\def\today{\number\day\space\ifcase\month\or   January\or February\or
   March\or April\or May\or June\or   July\or August\or September\or
   October\or November\or December\fi\   \number\year}

\newcommand{\ZZ}{{\mathbb{Z}}}
\newcommand{\NN}{{\mathbb{N}}}
\newcommand{\CC}{{\mathbb{C}}}
\newcommand{\RR}{{\mathbb{R}}}
\newcommand{\id}{{\mathrm{id}}}
\newcommand{\red}{{\mathrm{red}}}
\newcommand{\ca}{$C^*$-algebra}
\newcommand{\andSep}{\,\,\,\text{ and }\,\,\,}

\DeclareMathOperator{\Res}{Res}
\DeclareMathOperator{\GL}{GL}
\DeclareMathOperator{\QT}{QT}
\DeclareMathOperator{\Cu}{Cu}

\title{Strict comparison for twisted group $\mathrm{C}^*$-algebras}

\author{Sven Raum, Hannes Thiel, Eduard Vilalta}

\address{Sven~Raum,
Institute of Mathematics, University of Potsdam, Karl-Liebknecht-Str. 24-25, 14476 Potsdam, Germany}
\email{sven.raum@uni-potsdam.de}
\urladdr{https://raum-brothers.eu/sven/}

\address{Hannes~Thiel, 
Department of Mathematical Sciences, Chalmers University of Technology and the University of Gothenburg, SE-412 96 Gothenburg, Sweden}
\email{hannes.thiel@chalmers.se}
\urladdr{www.hannesthiel.org}

\address{Eduard Vilalta, 
Departament de Matem\`{a}tiques, Universitat Polit\`{e}cnica de Catalunya - BarcelonaTech (UPC), Diagonal 647, 08028 Barcelona, Spain}
\email[]{eduard.vilalta@upc.edu}
\urladdr{www.eduardvilalta.com}

\thanks{
EV was partially supported by MINECO (grant no.\  PID2023-147110NB-I00) and by the Comissionat per Universitats i Recerca de la Ge\-ne\-ralitat de Ca\-ta\-lu\-nya (grant no.\ 2021 SGR 01015).
HT and EV were partially supported by the Knut and Alice Wallenberg Foundation (KAW 2021.0140).
SR was partially supported by the German Research Foundation (DFG project no.\ 550184791).
}

\subjclass[2010]%
{Primary
46L05; 
Secondary
19K14, 
46L80, 
46L85. 
}
\date{\today}

\begin{document}

\begin{abstract}
We prove that the reduced twisted group \ca{} of any selfless group with the rapid decay property is selfless. 
As an application, we show that twisted group \ca{s} of acylindrically hyperbolic groups (possibly with nontrivial finite radical) and rapid decay are pure, and hence have strict comparison.
\end{abstract}

\maketitle

\section{Introduction}

Strict comparison is a regularity property for \ca{s} originally introduced by Blackadar \cite{Bla88Comparison} as a way to capture an analogue of the classical \emph{comparison theorem} for projections in $\mathrm{II}_1$-factors in the setting of simple \ca{s}. 
Since its inception, strict comparison has played a central role in the Elliott classification program and the renowned Toms-Winter conjecture \cite{MatSat12StrComparison,Win18ICM,Thi20RksOps}.
Strict comparison and various incarnations of it also play an important role in the analysis of non-simple \ca{s}. 
These generalizations have found applications in diverse areas, such as time-frequency analysis \cite{BedEnsVel22SmoothOrbStrComp} and the study of topological dynamical systems \cite{Ker20DimCompAlmFin}.
By \emph{strict comparison} we always mean strict comparison of positive elements by quasitraces, and we refer to \cite[Paragraph~3.4]{AntPerThiVil24arX:PureCAlgs} for details.

More recently, strict comparison has attracted significant interest in the context of \ca{s} arising from combinatorial or algebraic data. 
This includes not only the simple, nuclear case ---where comparison properties are tightly connected to classification results--- but also the simple, non-nuclear setting \cite{AmrGaoElaPat24arX:StrCompRedGpCAlgs, KunScha25arX:NegTarski, Vig25arX:StPropHighRankLat, HayKunRob25arX:SelflessRedFreeProd}, as well as certain non-simple but nuclear algebras~\cite{EnsVil25arX:ZstableTwGp,HirRorWin07CXAlgStabSSA,RobTik17NucDimNonSimple}.

Because strict comparison itself is not stable under many common constructions, one often works with stronger but more robust properties. 
These include notions like pureness~\cite{Win12NuclDimZstable, AntPerThiVil24arX:PureCAlgs}
and selflessness~\cite{Rob25aim:Selfless}, which behave well under various operations and imply strict comparison. 
A simplified diagram illustrating the relationships among these regularity properties is presented below.

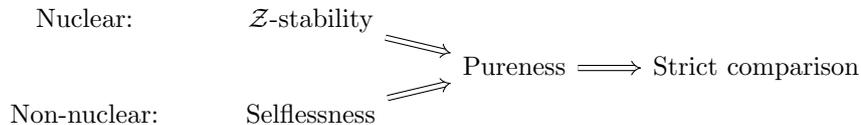
\begin{center}
    \begin{tikzcd}[row sep=2pt,arrows=Rightarrow]
        \text{Nuclear:}& \text{$\mathcal{Z}$-stability} \arrow[dr] & \\
         & & \text{Pureness} \arrow[r] & \text{Strict comparison} & \\
        \text{Non-nuclear:}& \text{Selflessness} \arrow[ru] & 
    \end{tikzcd}
\end{center}

This paper focuses on reduced group \ca{s} $C^*_\red(G)$ and their twisted counterparts $C^*_\red(G,\sigma)$ for non-amenable groups~$G$. 
In the untwisted case, a recent breakthrough in \cite{AmrGaoElaPat24arX:StrCompRedGpCAlgs} shows that a broad class of simple group \ca{s} are \emph{selfless}, and therefore pure and in particular possess strict comparison. 
Specifically, these results apply to reduced group \ca{s} of \emph{selfless} groups with the rapid decay property ---which include all finitely generated, acylindrically hyperbolic groups with trivial finite radical and rapid decay by \cite[Theorem~3.3]{AmrGaoElaPat24arX:StrCompRedGpCAlgs}.  
See also \cite{Vig25arX:StPropHighRankLat}, which proves that cocompact lattices in $\mathrm{PSL}(n, K)$ for $K$ a local field are selfless, which in characteristic zero combines with results on the rapid decay property \cite{RamRobSte98haagerupinequality,Lafforgue00propertyrd}.

A recent trend in the theory of twisted group \ca{s} has been to extend structural results from the untwisted setting to the twisted one; 
see \cite{BryKen18ResTwCrProd, AusRau26jimj:DetectIdlsRedCrProd, Rau24arX:TwAcylHypGps}. 
When paired with classical decomposition results, these extensions not only provide new results in the twisted case, but can also be used to cover more general classes of untwisted reduced group \ca{s}; 
see, for example, \cite{GerOsi20InvGpCAlgAcylHyperbolic} and \cite{Rau24arX:TwAcylHypGps}.
In this note, we continue this line of inquiry by proving that reduced twisted group \ca{s} of selfless groups with rapid decay are also selfless. 
We use this result to show that any reduced twisted group \ca{} over a finitely generated, acylindrically hyperbolic group with rapid decay is pure. 
This is new even in the untwisted setting, since our result removes the assumption of having trivial finite radical from \cite[Theorem~3.3]{AmrGaoElaPat24arX:StrCompRedGpCAlgs} at the expense of passing from selflessness to pureness.
As observed in \cite[Note after Theorem~B]{AmrGaoElaPat24arX:StrCompRedGpCAlgs}, the assumption of trivial finite radical is in fact necessary for selflessness of a reduced group \ca{}.

\begin{thmIntro}
\label{prp:ThmA}
Let $G$ be a selfless group with rapid decay and let $\sigma\in Z^2(G,\mathbb{T})$. 
Then the twisted reduced group \ca{} $C^*_\red(G,\sigma)$ is selfless.
\end{thmIntro}

As a consequence of \cref{prp:ThmA}, all \ca{s} covered in its statement are simple, separable, unital, and pure (in particular, they have strict comparison), with stable rank one and with a unique tracial state which is also the unique quasitracial state.
This leads to a computation of the Cuntz semigroup $\Cu(C^*_\red(G,\sigma))$ as $V(C^*_\red(G,\sigma)) \sqcup (0,\infty]$;
see \cref{prp:MainCor}.

\begin{thmIntro}
\label{prp:ThmB}
Let $G$ be a countable, acylindrically hyperbolic group with rapid decay, and let $\sigma\in Z^2(G,\mathbb{T})$. 
Then $C^*_\red(G,\sigma)$ is pure. 
In particular, it has strict comparison.
\end{thmIntro}

It is worth noting that \cref{prp:ThmB} covers all non-elementary, hyperbolic groups, also allowing for a nontrivial finite radical;
see \cref{prp:Hyperbolic}.
As a concrete example, we obtain that the reduced group \ca{} $C^*_\red(\GL(2,\ZZ))$ is pure.
Indeed, the group $\GL(2,\mathbb{Z})$ is virtually free and therefore hyperbolic, but its centre witnesses a nontrivial finite radical and so it was not covered by the results of \cite{AmrGaoElaPat24arX:StrCompRedGpCAlgs}.  As already pointed out, this is due to the fact that a selfless C*-algebra necessary is simple, underscoring the benefit of working with the more flexible concept of pureness in order to obtain results without any assumptions on the finite radical. We discuss further examples of pure twisted group C*-algebras in \cref{exa:twists} and \cref{exa:lattices-sl}.

After the authors made this work publicly available, further work on selflessness and pureness of group C*-algebras has appeared. We highlight the combination of Ozawa's work \cite{Oza25arX:Proximality}, establishing selflessness of all acylindrically hyperbolic groups with trivial finite radical, with work in \cite{FloKliOCPa26arX:Pureness}, which--based on Ozawa's work--establishes pureness of reduced group C*-algebras of arbitrary acylindrically hyperbolic groups.

\subsection*{Acknowledgments} 
This work was carried out when SR visited HT and EV at Chalmers University of Technology. They are grateful to the institution for its hospitality. EV and HT also wish to thank David Jekel for discussions on the subject.  The authors thank Itamar Vigdorovich for helpful comments on a first version of this work.

\section{The proofs}

Given a subset $X$ in a group $G$, we use $B_X(n)$ to denote the set of elements in~$G$ that are the product of at most $n$ elements in $X \cup X^{-1}$.
Following \cite[Definition~3.1]{AmrGaoElaPat24arX:StrCompRedGpCAlgs}, we say that a group $G$ with finite generating set $X$ is \emph{selfless} if there exists a sub-exponential map $f\colon\NN\to\RR$ such that, for every $n\geq 1$, there is a group homomorphism $\varphi_n\colon G* \ZZ\to G$ such that
\begin{itemize}
\item[(i)] 
$\varphi_n$ is injective on $B_{X\cup\{a\}}(n)$, where $a$ denotes a fixed generator in $\ZZ$;
\item[(ii)] 
$\varphi_n\big( B_{X\cup\{a\}}(n) \big) \subseteq B_X(f(n))$;
\item[(iii)] 
$\varphi_n|_G=\id_G$.
\end{itemize}
    
\begin{proof}[Proof of \cref{prp:ThmA}]
Let $G$ be a selfless group with rapid decay and let $\sigma\in Z^2(G,\mathbb{T})$.
We need to show that $C^*_\red(G,\sigma)$ is selfless.

Using that $G$ is selfless, fix group homomorphism $\varphi_n\colon G* \ZZ\to G$ as in the definition above.
For each $n$, consider the pull-back $\sigma_n :=\varphi_n^*\sigma\in Z^2(G * \ZZ,\mathbb{T})$, defined by the formula $\sigma_n(x,y) = \sigma(\varphi_n(x), \varphi_n(y))$.  
Since $\varphi_n|_G = \id_G$, the restriction satisfies $\Res^{G * \ZZ}_G\sigma_n = \sigma = \Res^{G * \ZZ}_G (\sigma * 1)$, where $\sigma * 1 \in Z^2(G * \ZZ, \mathbb{T})$ denote the unique 2-cocycle restricting to $\sigma$ on $G$ and to the constant map $1$ on $\ZZ$.
As $\ZZ$ has cohomological dimension $1$, the Mayer-Vietoris sequence calculating cohomology of free products (see e.g. \cite[Chapter~VII.9]{Bro82CohomologyGps} for the homological counterpart), shows that the restriction $\Res^{G * \ZZ}_G$ is an isomorphism from $H^2(G*\ZZ,\mathbb{T})$ to $H^2(G,\mathbb{T})$.  
So $\sigma_n$ is cohomologuous to $\sigma * 1$, and we let $\xi_n\colon G*\mathbb{Z}\to \mathbb{T}$ be a normalised function witnessing this, that is
\begin{equation}
\label{eq:xi-n}
\xi_n(xy) \cdot (\sigma* 1)(x,y)
= \xi_n(x) \cdot \xi_n(y) \cdot \sigma_n(x,y)
\quad\text{for all $x,y \in G\ast\ZZ$}
\end{equation}
and $\xi_n(e)=1$.

The twisted group algebra $\CC[G,\sigma]$ is the vector space of finite sums $\sum_{g \in G} a_g u_g$ with $a_g \in \CC$, and with (twisted) multiplication and involution induced by the rules
\[
u_g u_h = \sigma(g,h) u_{gh}, \andSep
(u_g)^\ast = \overline{\sigma(g,g^{-1})} u_{g^{-1}}
\quad\text{for $g,h \in G$}.
\]
Let $\pi_n\colon \CC[G*\ZZ,\sigma* 1]\to\CC[G,\sigma]$ be the linear map satisfying
\[
\pi_n (u_x)
=\xi_n(x)u_{\varphi_n (x)}
\quad\text{for $x\in G\ast\ZZ$}.
\]

We claim that $\pi_n$ is a $\ast$-homomorphism.
By linearity, it suffices to verify this on basis elements.
To see that $\pi_n$ is involutive, let $x \in G\ast\ZZ$.
Then
\[
\pi_n(u_x^*)
= \pi_n\Big( \overline{(\sigma \ast 1)(x,x^{-1})} u_{x^{-1}} \Big)
= \overline{(\sigma \ast 1)(x,x^{-1})} \xi_n(x^{-1}) u_{\varphi_n(x^{-1})},
\]
and
\[
\pi_n(u_x)^*
= \overline{\xi_n(x)} \ \overline{\sigma(\varphi_n(x),\varphi_n(x)^{-1})}
u_{\varphi_n(x)^{-1}}
= \overline{\xi_n(x)} \ \overline{\sigma_n(x,x^{-1})}
u_{\varphi_n(x^{-1})},
\]
and these agree using \eqref{eq:xi-n}. 

To show multiplicativity, let $x,y \in G \ast \ZZ$. 
Then
\[
\pi_n(u_xu_y)
= \pi_n\big( (\sigma \ast 1 )(x,y) u_{xy} \big)
=\xi_n(xy) (\sigma \ast 1 )(x,y) u_{\varphi_n(xy)}
\]
and
\begin{align*}
\pi_n(u_x)\pi_n(u_y)
&= \xi_n(x) u_{\varphi_n(x)} \xi_n(y) u_{\varphi_n(y)} \\
& =\xi_n(x) \xi_n(y) \sigma\big( \varphi_n(x),\varphi_n(y) \big) u_{\varphi_n(x)\varphi_n(y)} \\
& =\xi_n(x) \xi_n(y) \sigma_n(x,y) u_{\varphi_n(xy)}
\end{align*}
and again these agree by \eqref{eq:xi-n}.

The rest of the proof follows as in \cite[Theorem~3.5, Proposition~3.7]{AmrGaoElaPat24arX:StrCompRedGpCAlgs}.
\end{proof}

\begin{cor}
\label{prp:MainCor}
Let $G$ be a selfless group with rapid decay and let $\sigma\in Z^2(G,\mathbb{T})$. 
Then $C^*_\red(G,\sigma)$ is simple, separable, unital, pure (in particular, it has strict comparison), with stable rank one, and with a unique tracial state which is also the unique quasitracial state.
The Cuntz semigroup can be computed as
\[
\Cu(C^*_\red(G,\sigma)) \cong V(C^*_\red(G,\sigma)) \sqcup (0,\infty].
\]
\end{cor}
\begin{proof}
It is clear that $C^*_\red(G,\sigma)$ is separable and unital.
Further, by \cite[Theorem~3.1]{Rob25aim:Selfless}, $C^*_\red(G,\sigma)$ is simple, has stable rank one, has strict comparison, and has a unique tracial state, which is also the unique quasitracial state.
Consequently, $C^*_\red(G,\sigma)$ is pure, for example by \cite[Theorem~5.13]{AntPerThiVil24arX:PureCAlgs}.

An element in the Cuntz semigroup of a simple, stably finite \ca{} $A$ is either compact, or nonzero and soft;
see \cite[Proposition~5.3.16]{AntPerThi18TensorProdCu}.
The compact elements form a submonoid that is naturally isomorphic to the Murray-von Neumann semigroup $V(A)$.
If $A$ is also unital, separable and pure, then the nonzero, soft elements form a subsemigroup that is naturally isomorphic to the semigroup of lower-semicontinuous, affine functions $\QT(A)\to(0,\infty]$, where $\QT(A)$ denotes the Choquet simplex of quasitracial states;
see \cite[Section~7.3]{AntPerThi18TensorProdCu}.
Now, the computation of $\Cu(C^*_\red(G,\sigma))$ follows using that $\QT(C^*_\red(G,\sigma))$ is a singleton.
\end{proof}

\begin{lma}
\label{prp:TrivialFinRad}
Let $G$ be a group with trivial finite radical, and let $H \subseteq G$ be a subgroup of finite index.
Then $H$ has trivial finite radical.
\end{lma}
\begin{proof}
We first recall a useful characterization of the finite radical:
\emph{An element~$x$ in a group~$K$ belongs to the finite radical if and only if $x$ has finite order and finitely many conjugates.}
The forward implication is obvious.
For the backward implication, assume that $x \in K$ has finite order and finitely many conjugates.
Then the set
\[
X := \big\{ yx^ny^{-1} : y \in K, n \in \ZZ \big\}
\]
is finite and it contains all conjugacy classes of its elements.
Then it follows from Dicman's Lemma (\cite[14.5.7]{Rob96CourseThyGps2ed}) that $X$ generates a finite normal subgroup of~$K$, which shows that $x$ belongs to the finite radical of $K$.

\smallskip

We now prove the lemma.
Let $x \in H$ belong to the finite radical of $H$.
Then~$x$ has finite order and finitely many conjugates in~$H$.
Since $H$ has finite index in $G$, it follows that $x$ has also finitely many conjugates in~$G$.
Using the above observation, we see that $x$ belongs to the finite radical of $G$, and thus $x=e$, showing that $H$ has trivial finite radical.
\end{proof}

\begin{proof}[Proof of \cref{prp:ThmB}]
Assume first that $G$ is finitely generated, that is, let $G$ be a finitely generated, acylindrically hyperbolic group with rapid decay, and let $\sigma\in Z^2(G,\mathbb{T})$. 
We show that $C^*_\red(G,\sigma)$ is pure. 

Proceeding as in the proof of \cite[Theorem~1]{Rau24arX:TwAcylHypGps}, let $K$ be the finite radical of~$G$. 
Then, using \cite[Theorem~2.1]{Bed91GpsSimpleCAlgs}, the twisted group \ca{} $C^*_\red(G,\sigma)$ is isomorphic to the reduced twisted crossed product $\CC[K,\sigma]\rtimes_{\alpha,\rho,\red} G/K$, where we denote by $\sigma$ also its restriction to $K$. 
Since $\CC[K,\sigma]$ is finite-dimensional, it follows from \cite[Theorem~2.13]{Gre80StrImprimitivity} that this twisted crossed product can be written as a finite direct sum where each summand is a matrix amplification of a reduced twisted group \ca{} over a finite-index subgroup of $G/K$.

Note that $G/K$ has trivial finite radical, and hence, by \cref{prp:TrivialFinRad}, so does every finite-index subgroup of $G/K$.
Since $G/K$ is finitely generated, Schreier's Lemma implies that all its finite-index subgroups are also finitely generated.
Further, acylindric hyperbolicity passes to quotients by finite, normals subgroups, as well as to finite-index subgroups by \cite[Lemma~3.9]{MinOsi15AcylHypGpsTrees} (see also \cite[Lemma~1]{MinOsi19CorrAcylHypGpsTrees}).
Thus, every finite-index subgroup of~$G/K$ is finitely generated, acylindrically hyperbolic and has trivial finite radical, and is therefore selfless by \cite[Theorem~3.3]{AmrGaoElaPat24arX:StrCompRedGpCAlgs}.
 
Next, by \cite[Propositions~2.1.1, 2.1.4]{Jol90RapDecrRedGp} the rapid decay property passes to quotients by finite, normal subgroups, as well as to subgroups.
Summarizing, finite-index subgroups of $G/K$ are selfless and have rapid decay, and therefore their twisted, reduced group \ca{s} are pure by \cref{prp:MainCor}.

It follows that $C^*_\red(G,\sigma)$ is a finite direct sum of matrix amplifications of pure \ca{s}, and hence pure itself.

Now let $G$ be any countable, acylindrically hyperbolic group with rapid decay.  By \cite[Theorem 1.2]{Osi16AcylHypGps}, there is a non-elementary acylindrical action of $G$ on a hyperbolic space $X$ and $G$ is not virtually cyclic.  So by \cite[Theorem 1.1]{Osi16AcylHypGps}, there are infinitely many pairwise independent loxodromic elements in $G$.  Fix two such elements $g,h \in G$.  Given an arbitrary finite subset $F \subseteq G$, the group $H = \langle F \cup \{g,h\} \rangle$ acts acylindrically hyperbolic on $X$ and its limit set contains the limit set of $\langle g,h \rangle$.  So the action of $H$ on $X$ is non-elementary.  This shows that $H$ is acylindrically hyperbolic.  Since rapid decay passes to arbitrary subgroups by \cite[Proposition 2.1.1]{Jol90RapDecrRedGp}, we have shown that $G$ is the directed union of finitely generated, acylindrically hyperbolic subgroups with rapid decay. Using that pureness passes to inductive limits \cite[Theorem~D]{PerThiVil25arX}, the result follows from the finitely generated case.
\end{proof}

\begin{cor}
\label{prp:Hyperbolic}
Let $G$ be a non-elementary, hyperbolic group and let $\sigma\in Z^2(G,\mathbb{T})$. 
Then $C^*_\red(G,\sigma)$ is pure.
\end{cor}
\begin{proof}
By definition, every hyperbolic group is finitely generated.
As noted in \cite[Appendix~8]{Osi16AcylHypGps}, non-elementary, hyperbolic groups are also acylindrically hyperbolic.
Further, hyperbolic groups enjoy the rapid decay property by \cite{Jol90RapDecrRedGp, Har88GpsHyperboliques}.  
Thus, \cref{prp:ThmB} applies.
\end{proof}

While free groups do not have any non-trivial twisted group C*-algebras due to the fact that $\mathrm{H}^2(\mathbb{F}_n, \mathbb{T})$ vanishes, many other groups to which \cref{prp:ThmA} applies have non-trivial second cohomology with coefficients in the circle group.
\begin{exa}
  \label{exa:twists}
  Denote by $\Sigma_g$ the closed surface of genus $g \geq 2$ and $G = \pi_1(\Sigma_g)$ the associated surface group.  Then $G$ is torsion-free hyperbolic.  Its homology groups are $\mathrm{H}_0(G, \ZZ) \cong \ZZ$, $\mathrm{H}_1(G, \ZZ) \cong \ZZ^{2g}$ and $\mathrm{H}_2(G, \ZZ) \cong \ZZ$.  So the universal coefficient theorem yields an isomorphism $\mathrm{H}^2(G, \mathbb{T}) \cong \mathrm{Hom}(\mathbb{Z}, \mathbb{T}) \cong \mathbb{T}$.  For $\theta \in \mathbb{T}$, we denote by $\sigma_\theta \in \mathrm{H}^2(G, \mathbb{T})$ the unique 2-cocycle mapping to $\theta$.  Then $C^*_\red(G, \sigma_\theta)$ is selfless by \cref{prp:ThmA}.
\end{exa}
The next example shows how Vigdorovich's work on cocompact lattices in projective special linear groups \cite{Vig25arX:StPropHighRankLat} can be lifted to special linear groups.  We remark that cocompact lattices in $\mathrm{SL}(3, K)$ can have a non-trivial centre, so that $C^*_\red(G)$ cannot be simple, hence not selfless.
\begin{exa}
  \label{exa:lattices-sl}
  Let $K$ be a local field of characteristic zero and $G \leq \mathrm{SL}(3,K)$ a cocompact lattice.  Given any $\sigma \in \mathrm{H}^2(G, \mathbb{T})$, we show that $C^*_\red(G, \sigma)$ is pure.  Denote by $\mathrm{Ad}\colon \mathrm{SL}(3,K) \longrightarrow \mathrm{PSL}(3,K)$ the adjoint representation.  The kernel of $\mathrm{Ad}$ equals the centre of $\mathrm{SL}(3,K)$.  It has order at most $3$, since there are at most three third roots of unity in $K$.  Write $Z = \ker(\mathrm{Ad}|_G)$.  Then $Z$ is a finite, central subgroup of $G$.  The quotient $\mathrm{Ad}(G) \cong G/Z$ and all its finite-index subgroups are cocompact lattices in $\mathrm{PSL}(3,K)$. So by \cite[Theorem 1.2 and Lemma 2.3]{Vig25arX:StPropHighRankLat}, they are selfless.  Further, they have the rapid decay property by \cite{RamRobSte98haagerupinequality, Lafforgue00propertyrd}.  We can now argue as in the proof of \cref{prp:ThmB}.
\end{exa}



\providecommand{\bysame}{\leavevmode\hbox to3em{\hrulefill}\thinspace}
\providecommand{\noopsort}[1]{}
\providecommand{\mr}[1]{\href{http://www.ams.org/mathscinet-getitem?mr=#1}{MR~#1}}
\providecommand{\zbl}[1]{\href{http://www.zentralblatt-math.org/zmath/en/search/?q=an:#1}{Zbl~#1}}
\providecommand{\jfm}[1]{\href{http://www.emis.de/cgi-bin/JFM-item?#1}{JFM~#1}}
\providecommand{\arxiv}[1]{\href{http://www.arxiv.org/abs/#1}{arXiv~#1}}
\providecommand{\doi}[1]{\url{http://dx.doi.org/#1}}
\providecommand{\MR}{\relax\ifhmode\unskip\space\fi MR }
\providecommand{\MRhref}[2]{%
  \href{http://www.ams.org/mathscinet-getitem?mr=#1}{#2}
}
\providecommand{\href}[2]{#2}

\end{document}